\documentclass[11pt]{article}
\usepackage[a4paper,margin=29mm]{geometry}
\usepackage[T1]{fontenc}
\usepackage{lmodern}
\usepackage{amsmath,amssymb,amsthm,mathtools}
\usepackage{booktabs,array,microtype}
\usepackage[hidelinks]{hyperref}
\usepackage[nameinlink,capitalise]{cleveref}

\hypersetup{
  pdftitle={A two-step supercongruence for an Ap\'ery-like sequence},
  pdfauthor={Huimin Zheng},
  pdfsubject={A proof of Zhi-Hong Sun's Conjecture 2.6},
  pdfkeywords={Apery-like numbers, supercongruences, p-adic gamma function,
    Dwork congruences, finite-field hypergeometric functions}
}

\newtheorem{theorem}{Theorem}[section]
\newtheorem{proposition}[theorem]{Proposition}
\newtheorem{lemma}[theorem]{Lemma}

\theoremstyle{definition}
\theoremstyle{remark}\newtheorem{remark}[theorem]{Remark}

\newcommand{\Zp}{\mathbb Z_p}
\newcommand{\Qp}{\mathbb Q_p}
\newcommand{\Fq}{\mathbb F_q}
\newcommand{\OO}{\mathcal O}
\newcommand{\Aalg}{\mathcal A}
\newcommand{\LL}{\mathcal L}
\newcommand{\JJ}{\mathcal J}
\newcommand{\TT}{\mathsf T}
\newcommand{\vp}{v_p}
\newcommand{\fall}[1]{^{\underline{#1}}}

\title{A two-step supercongruence for an Ap\'ery-like sequence}
\author{Huimin Zheng\\
\small Anhui Science and Technology University\\
\small \texttt{zhenghm@ahstu.edu.cn}}
\date{}

\begin{document}
\maketitle

\begin{abstract}
Let
\[
G_n=\sum_{k=0}^n4^k\binom{2n-2k}{n-k}^{2}\binom{2k}{k}.
\]
Zhi-Hong Sun conjectured that, for primes $p\equiv3\pmod4$, positive
odd integers $m$, and $r\ge2$, the two-step congruence
$G_{(mp^r-1)/2}\equiv p^2G_{(mp^{r-2}-1)/2}\pmod {p^{2r-1}}$ holds.
We prove the stronger valuation statement
\[
G_{(p^2M-1)/2}-p^2G_{(M-1)/2}\in
p^{2v_p(M)+3}\mathbb Z_p
\]
for every positive odd $M$.  The proof converts the sum to a terminating
${}_3F_2$, constructs a cancelled digit-transfer operator, and identifies
a two-dimensional analytic quotient of its cubic difference operator.
The quotient operator has characteristic polynomial $X^2-p^2$; its second
trace is evaluated through the Gross--Koblitz formula and Greene's
finite-field Dixon identity.  A logarithmic-loss Green inverse converts
this spectral identity into an actual integral analytic primitive.  The
exceptional prime $3$ requires a finite exact PARI/GP certificate, while
the infinite tail is bounded symbolically.  Thus the computation is finite,
reproducible, and separated from the uniform part of the proof.
\end{abstract}



\section{Introduction}

The sequence $(G_n)_{n\ge0}$ begins $1,12,164,2352,34596,\ldots$ and is
one of the sporadic Ap\'ery-like sequences considered by Sun
\cite{Sun2020}.  The following was stated there as Conjecture~2.6.

\begin{theorem}\label{thm:main}
Let $p\equiv3\pmod4$ be prime, let $m$ be a positive odd integer, and let
$r\ge2$.  Then
\[
 G_{(mp^r-1)/2}\equiv p^2G_{(mp^{r-2}-1)/2}\pmod {p^{2r-1}}.
\]
More precisely, if $M$ is any positive odd integer, then
\begin{equation}\label{eq:strong-main}
 G_{(p^2M-1)/2}-p^2G_{(M-1)/2}
 \in p^{2\vp(M)+3}\Zp .
\end{equation}
\end{theorem}

The residue class $p\equiv3\pmod4$ is the supersingular branch of the
underlying hypergeometric family.  Ordinary unit-root congruences do not
directly apply there.  First-layer CM results such as
\cite{KibelbekLongMossShellerYuan} explain the initial vanishing, while
the arbitrary-length theory of \cite{BeukersDelaygue} assumes a unit root.
The present proof instead packages two base-$p$ digits at a time.

Our argument has four components.  First, an exact binomial transform puts
the original convolution into a terminating hypergeometric form.  Second,
cancelled digit weights define an analytic transfer operator $T_N$ and a
cubic difference operator $D_N$.  Third, the pure transfer $T=T_0$ acts on
a two-dimensional Hausdorff quotient.  Its characteristic polynomial is
$X^2-p^2$, obtained from exact first and second traces.  Finally, a Green
kernel with only logarithmic denominator loss upgrades the quotient identity
to an actual analytic primitive.  At $p=3$ its required $3^4$-integrality is
reduced to a finite coefficient calculation with a rigorously controlled
infinite tail.

For orientation, the logical dependence is
\[
 \begin{aligned}
 \text{endpoints}&\Longrightarrow
 \text{coefficient decay and quotient spectrum},\\
 &\Longrightarrow\text{an actual primitive}
 \Longrightarrow\text{all-depth descent}.
 \end{aligned}
\]
The middle implication is the point at which topology matters.  The spectral
calculation gives only membership in the closure of the image of the
difference operator.  A separately constructed Green inverse, together with
the coefficient decay, produces an actual convergent primitive.  We keep the
main line in the body of the paper and place the longer finite-product and
denominator calculations in Appendices~\ref{app:large-prime}--\ref{app:p3}.

The finite-field part uses classical results of Gross--Koblitz
\cite{GrossKoblitz} and Greene \cite{Greene}.  A recent coefficient bound for
the scaled Morita gamma function is taken from
\cite{BaiPomerleanoSeidel}; no result concerning quantum connections is used.
The modular-form occurrence of $G_n$ studied in \cite{AllenLongSaad} is
structurally suggestive, but its Fourier-index recurrence does not directly
match the affine half-index tower in \cref{thm:main}.

\section{The terminating hypergeometric representation}

For a positive odd integer $N$, put
\[
 H(N)=G_{(N-1)/2},\qquad \TT(N)=4^{1-N}H(N),\qquad L_N=(N-1)/2,
\]
and define
\begin{equation}\label{eq:BJN}
 B_k(N)=\frac{((1-N)/2)_k(1/2)_k^2}{(k!)^3}
 =(-1)^k\binom{L_N}{k}\frac{\binom{2k}{k}^2}{16^k},
 \qquad \JJ_N(f)=\sum_{k=0}^{L_N}B_k(N)f(k).
\end{equation}

\begin{lemma}\label{lem:representation}
For every positive odd $N$, one has $\TT(N)=\JJ_N(1)$.
\end{lemma}
\begin{proof}
Set $a_k=\binom{2k}{k}^2/16^k$ and
$A(z)=\sum_{k\ge0}a_kz^k$.  Directly from the definition of $G_n$,
\[
 \sum_{n\ge0}\frac{G_n}{16^n}z^n=(1-z)^{-1/2}A(z).
\]
The alternating binomial transform of $(a_k)$ has generating function
$(1-z)^{-1}A(-z/(1-z))$.  For $z$ near zero,
\[
 A(z)=\frac1\pi\int_0^1
 \frac{dx}{\sqrt{x(1-x)}\sqrt{1-zx}}.
\]
Replacing $z$ by $-z/(1-z)$ and then $x$ by $1-x$ yields
$A(-z/(1-z))=(1-z)^{1/2}A(z)$.  Coefficient comparison gives the claim.
\end{proof}

Fix an odd prime $p$.  Let $\OO=\Zp\langle a\rangle$ be the integral
Tate algebra, with its coefficient Gauss norm, and put
\[
 Df=a^3f(a)-(a+1/2)^3f(a+1),\qquad
 B_2f=(a+1/2)^2f(a+1),\qquad D_N=D+\frac N2B_2.
\]
Since the coefficients in \eqref{eq:BJN} are integral,
\begin{equation}\label{eq:Jintegral}
 \JJ_N(\OO)\subseteq\Zp.
\end{equation}

\section{Cancelled digit transfer}

Let $h=(p-1)/2$ and
$\boldsymbol\alpha(N)=((1-N)/2,1/2,1/2)$.  For integers $a\ge0$, set
\begin{equation}\label{eq:VN}
 V_N(a)=
 \frac{\displaystyle\prod_{i=1}^3
  \prod_{\substack{0\le j<pa\\j\not\equiv h\ (p)}}
       (\alpha_i(N)+j)}
 {\displaystyle\left(\prod_{\substack{1\le j\le pa\\p\nmid j}}j\right)^3},
\end{equation}
and, for $0\le d<p$,
\[
\begin{aligned}
 g_{N,d}(x)&=\frac{\prod_{i=1}^3(x+\alpha_i(N))_d}{(x+1)_d^3},\\
 w_{N,d}(a)&=V_N(a)g_{N,d}(pa),\\
 (T_Nf)(a)&=\sum_{d=0}^{p-1}w_{N,d}(a)f(pa+d).
\end{aligned}
\]
The cancellation in \eqref{eq:VN} is performed before analytic continuation;
one never divides by a terminating zero.

\begin{lemma}[Exact endpoints]\label{lem:endpoints}
For every positive odd $M$,
\begin{equation}\label{eq:endpoints}
 \JJ_{pM}=\JJ_MT_{pM},\qquad \JJ_MD_M=0.
\end{equation}
\end{lemma}
\begin{proof}
Removing from each numerator Pochhammer product the factors with indices
$h+p\ell$ and from each factorial the factors $p,2p,\ldots,pa$ gives
\[
 B_{pa+d}(N)=B_a(N/p)V_N(a)g_{N,d}(pa).
\]
For $N=pM$, $L_N=pL_M+h$.  In the terminal row $a=L_M$, every digit
$d>h$ contains the zero factor $p(a+1/2-M/2)$.  Grouping by $k=pa+d$
proves the first identity.  For the second, use
\[
 B_{k+1}(M)(k+1)^3
 =B_k(M)(k+1/2-M/2)(k+1/2)^2
\]
and telescope; both boundary terms vanish.
\end{proof}

We use the scaled Morita gamma estimate proved in
\cite[Sec.~2]{BaiPomerleanoSeidel}.  In the normalization below, applying
their coefficient estimate to the scaled argument $-pz$ gives the explicit
slope
\begin{equation}\label{eq:gamma-slope}
 \Gamma_p(-pz)=\sum_{m\ge0}e_mz\fall m,\qquad
 \vp(e_m)\ge\alpha_pm,\qquad
 \alpha_p=\frac{p^2-3p+1}{p(p-1)}>0.
\end{equation}
Here $z\fall m=z(z-1)\cdots(z-m+1)$.  The gamma functional equation,
applied after the cancellation in \eqref{eq:VN}, proves the following.

\begin{proposition}[Bidisc estimate]\label{prop:bidisc}
For $N=pU$, every weight has a jointly analytic expansion
\[
 w_{pU,d}(a)=\sum_{r,s\ge0}u_{d;r,s}a^rU^s,\qquad
 \vp(u_{d;r,s})\ge\alpha_p(r+s).
\]
The estimate holds in the two-variable Tate algebra and is preserved by
integral affine substitutions and inversion of the unit factors occurring
in \eqref{eq:VN}.
\end{proposition}
\begin{proof}
Fix $0<\delta<\alpha_p$ and use the weighted Gauss norm on
$|a|,|U|\le p^\delta$.  Falling factorials have integral coefficients,
so \eqref{eq:gamma-slope} makes each shifted scaled gamma factor a unit
congruent to $1$ modulo the maximal ideal.  Products, inverses and the
finitely many digit factors all have norm at most one.  Cauchy's coefficient
estimate gives $\vp(u_{d;r,s})\ge\delta(r+s)$; let
$\delta\uparrow\alpha_p$.
\end{proof}

Put $T=T_0$, $S=T1$, and
\[
 \Lambda=[N]T_N=p^{-1}[U]T_{pU},\qquad
 \LL=\Zp+p\OO.
\]
Here $[x]F(x)$ denotes the coefficient of $x$ in the Taylor expansion of
$F$ at $x=0$; in particular $[N]T_N=\left.\partial T_N/\partial
N\right|_{N=0}$.
The next proposition contains the arithmetic input.  Its proof is split into
the quotient calculation of \cref{sec:quotient}, the integral primitive of
\cref{sec:primitive}, and the exceptional calculation of
\cref{sec:p3}.

\begin{proposition}[Pure two-digit input]\label{prop:pure-input}
If $p\equiv3\pmod4$, then
\begin{equation}\label{eq:pure-input}
 S\in p^2\OO,\qquad \Lambda\OO\subseteq\LL,
 \qquad (T^2-p^2)1=Dq\quad\text{for some }q\in p^4\OO.
\end{equation}
For $p\ge19$ one may take $q\in p^5\OO$.
\end{proposition}

\section{The analytic quotient and its exact spectrum}\label{sec:quotient}

Let $\Aalg=\Qp\langle a\rangle$ and $u=a+1/4$.  Polynomial division gives
\[
 \Qp[a]/D\Qp[a]=\Qp[1]\oplus\Qp[u].
\]
Write $f\equiv\gamma(f)+\beta(f)u\pmod{D\Qp[a]}$.
For $g_n=\gamma\binom an$ and $b_n=\beta\binom an$, one obtains
\begin{align}
 \sum_{n\ge0}g_nt^n&={}_2F_1(1/4,1/4;1/2;-t)^2,\label{eq:momgen1}\\
 \sum_{n\ge0}b_nt^n&=t\,{}_2F_1(3/4,3/4;3/2;-t)^2.\label{eq:momgen2}
\end{align}
Indeed, applying either quotient functional to $D(1+t)^a$ gives
\[
8t(1+t)^2F'''+12(1+t)(3t+1)F''+2(13t+9)F'+F=0,
\]
and \eqref{eq:momgen1}--\eqref{eq:momgen2} are the two solutions with
the required initial pairs.  Counting numerator and denominator roots at
each power of $p$ gives
\begin{equation}\label{eq:momentbound}
 \vp(g_n),\vp(b_n)\ge-2\lfloor\log_p(2n+1)\rfloor.
\end{equation}
Consequently the Newton expansions
\[
 \gamma(f)=\sum_{n\ge0}\Delta^nf(0)g_n,\qquad
 \beta(f)=\sum_{n\ge0}\Delta^nf(0)b_n
\]
converge for $f\in\Aalg$.  They define bounded functionals and the projection
$\Pi f=\gamma(f)+\beta(f)u$ satisfies
\begin{equation}\label{eq:quotient}
 \overline{D\Aalg}=\ker\Pi,\qquad
 \Aalg/\overline{D\Aalg}\simeq\Qp[1]\oplus\Qp[u].
\end{equation}

For the pure weights write
\[
 V(a)=-\left\{\frac{\Gamma_p(1/2+pa)}
 {\Gamma_p(1/2)\Gamma_p(1+pa)}\right\}^{\!3},
 \qquad (Kf)(a)=V(a)f(pa).
\]
Termwise telescoping gives the exact intertwining relation
\begin{equation}\label{eq:intertwine}
 TD=p^3DK.
\end{equation}
Thus $T$ induces an endomorphism $M_{p,\infty}$ of the quotient
\eqref{eq:quotient}.

\begin{proposition}[Exact quotient spectrum]\label{prop:spectrum}
For every prime $p\equiv3\pmod4$,
\[
 \operatorname{tr}M_{p,\infty}=0,\qquad
 \operatorname{tr}(M_{p,\infty}^2)=2p^2,\qquad
 M_{p,\infty}^2=p^2I.
\]
In particular,
\begin{equation}\label{eq:zeroquotient}
 \Pi((T^2-p^2)1)=0.
\end{equation}
\end{proposition}
\begin{proof}
Put $H=\ker\Pi=\overline{D\Aalg}$, $L=T|_H$, and $U=p^3K$.
We first justify the Fredholm comparison, since the image of $D$ is only
known to be dense in $H$.  The operator $D:\Aalg\to H$ is injective:
if $Df=0$, evaluation successively at $a=0,1,2,\ldots$ gives
$f(n)=0$ for every positive integer $n$, and the zeros $p^m\to0$ force
$f=0$.  Moreover, $T$, $K$, and $L$ are completely continuous.  Indeed,
$f(a)\mapsto f(pa)$ is the operator-norm limit of its finite-rank
truncations in the monomial basis, while translation and multiplication
by the analytic weights are bounded; for $L$ one composes the
finite-rank approximations to $T$ with $1-\Pi$ and restricts to $H$.

The intertwining relation \eqref{eq:intertwine} says $LD=DU$.  We claim
that
\begin{equation}\label{eq:fredholm-dense-comparison}
 \det(1-tL)=\det(1-tU).
\end{equation}
Let $\lambda\ne0$, extending scalars to a finite extension of $\Qp$ if
necessary.  If $\lambda-U$ is invertible, then
\[
 (\lambda-L)D\Aalg=D(\lambda-U)\Aalg=D\Aalg,
\]
so the range of $\lambda-L$ is dense.  By the nonarchimedean Fredholm
alternative it is also closed and has index zero; hence
$\lambda-L$ is invertible.  Conversely, a nonzero vector in
$\ker(\lambda-U)$ is carried by the injective map $D$ to one in
$\ker(\lambda-L)$.

The same argument preserves algebraic multiplicities.  Let
$\Aalg=E_U\oplus F_U$ and $H=E_L\oplus F_L$ be the Riesz decompositions
at $\lambda$, and let $P_L$ be the projection onto $E_L$.  If
$(L-\lambda)^mE_L=0$, then for $x\in F_U$, writing
$x=(U-\lambda)^my$, we obtain
\[
 P_LDx=(L-\lambda)^mP_LDy=0.
\]
Thus $P_LD\Aalg=D(E_U)$.  The left-hand side is dense in $E_L$, whereas
$D(E_U)$ is finite-dimensional and therefore closed.  Hence
$D:E_U\to E_L$ is an isomorphism intertwining the two restrictions, so
all Jordan multiplicities agree.  Serre's Proposition~12 and its
corollary \cite[pp.~80--82]{Serre} give the Riesz and Fredholm assertions
used here, and the normalized Fredholm determinant is determined by its
nonzero zeros with multiplicity as in the proof of his Proposition~15
\cite[p.~84]{Serre}.  This proves \eqref{eq:fredholm-dense-comparison}.
Applying Fredholm multiplicativity to the closed invariant subspace $H$
(Serre's Lemma~2 \cite[pp.~77--78]{Serre}) and then using the formal
logarithmic determinant identity (Serre's Proposition~7,
Corollary~3) gives
\begin{align}
 \det(1-tM_{p,\infty})
   &=\frac{\det(1-tT)}{\det(1-p^3tK)},\label{eq:fredholm-quotient}\\
 \operatorname{tr}(M_{p,\infty}^r)
   &=\operatorname{Tr}(T^r)-p^{3r}\operatorname{Tr}(K^r)
   \qquad(r\ge1).\label{eq:trace-quotient}
\end{align}
No bounded inverse for $D$ and no equality $D\Aalg=H$ is used.

We next record the fixed-point trace entering
\eqref{eq:trace-quotient}.  If
$Bf(a)=w(a)f(p^ra+d)$, its fixed point is
$a_*=-d/(p^r-1)$.  Translating $a_*$ to the origin makes the diagonal
entries in the monomial basis equal to $w(a_*)p^{rn}$, and hence
\begin{equation}\label{eq:weighted-fixed-point}
 \operatorname{Tr}B=\frac{w(a_*)}{1-p^r}.
\end{equation}
For a digit string ${\bf d}=(d_0,\ldots,d_{r-1})$, set
\[
 m({\bf d})=\sum_{i=0}^{r-1}p^{r-1-i}d_i,
 \qquad
 a_i=p^ia+\sum_{s=0}^{i-1}p^{i-1-s}d_s.
\]
The corresponding summand of $T^r$ has weight
$\prod_{i=0}^{r-1}w_{d_i}(a_i)$ and fixed point
$a=-m({\bf d})/(p^r-1)$, so \eqref{eq:weighted-fixed-point} gives the
trace of $T^r$ by summing over all digit strings.  For $K^r$ the fixed
point is zero and the product of pure weights is one, since $V(0)=1$;
therefore
\begin{equation}\label{eq:K-fixed-trace}
 \operatorname{Tr}(K^r)=\frac1{1-p^r}.
\end{equation}

We evaluate the first trace explicitly.  Write $h=(p-1)/2$, which is
odd, and put
\[
 U_0(a)=-\frac{\Gamma_p(a+1/2)}
 {\Gamma_p(1/2)\Gamma_p(a+1)},
 \qquad B_0(a)=(a+1/2)U_0(a),
 \qquad a_d=-\frac d{p-1}.
\]
The block cancellation in the definition of the weights and the Morita
functional equation give
\begin{equation}\label{eq:first-fixed-weights}
 w_d(a_d)=
 \begin{cases}
  U_0(a_d)^3,&0\le d\le h,\\
  p^3B_0(a_d)^3,&h<d\le2h.
 \end{cases}
\end{equation}
The only exceptional numerator factor is the factor with index $h$,
namely $p(a+1/2)$, and it occurs precisely for $d>h$; the denominator
has no exceptional factor because $d<p$.

Let $r(x)\in\{1,\ldots,p\}$ be the least positive residue, with
$r(0)=p$.  Morita reflection is
$\Gamma_p(x)\Gamma_p(1-x)=(-1)^{r(x)}$.  On $0\le d\le h$, pair
$d$ with $d'=h-d$.  Then $a_{d'}=-1/2-a_d$ and
\[
 \frac{U_0(a_{d'})}{U_0(a_d)}
 =(-1)^{r(-a_d)-r(1/2-a_d)}
 =(-1)^{(p-d)-(h+1-d)}=-1.
\]
On $h<d<2h$, pair $d$ with $d'=3h-d$.  Since the relevant arguments
are units, the functional equation and reflection give
\[
 \frac{B_0(a_{d'})}{B_0(a_d)}
 =-(-1)^{r(-a_d)-r(1/2-a_d)}
 =-(-1)^{(p-d)-(p+h+1-d)}=-1.
\]
Neither involution has a fixed point because $h$ is odd.  The only
unpaired term is $d=2h=p-1$, where $a_d=-1$,
$U_0(-1)=-2$, and $B_0(-1)=1$.  Thus
\begin{equation}\label{eq:first-fixed-sum}
 \sum_{d=0}^{p-1}w_d(a_d)=p^3.
\end{equation}
Equations \eqref{eq:weighted-fixed-point},
\eqref{eq:K-fixed-trace}, and \eqref{eq:trace-quotient} now yield
$\operatorname{tr}M_{p,\infty}=0$.

It remains to calculate the second trace.  Put
$q=p^2$, $N=q-1$, $H=N/2$, $c=\Gamma_p(1/2)$, and
\[
 z_d(a)=-\frac{\Gamma_p(pa+1/2)}{c\Gamma_p(pa+1)}
       \frac{(pa+1/2)_d}{(pa+1)_d},
 \qquad w_d(a)=z_d(a)^3.
\]
For $j=pd+e\in\{0,\ldots,N\}$ set
\[
 j'=pe+d,\qquad a=-\frac jN,\qquad
 b=pa+d=-\frac{j'}N,\qquad P_j=z_d(a)z_e(b).
\]
The fixed-point formula for the $p^2$-affine summands of $T^2$ gives
\begin{equation}\label{eq:T2-fixed-trace}
 \operatorname{Tr}(T^2)=\frac{\Sigma_{p,2}}{1-p^2},
 \qquad \Sigma_{p,2}:=\sum_{j=0}^{N}P_j^3.
\end{equation}
We now derive the finite-field expression for $P_j$, retaining all
carry and endpoint cases.

Let $\Gamma=\Gamma_p$ and
\[
 U(a)=-\frac{\Gamma(a+1/2)}{c\Gamma(a+1)},\quad
 u=a+\frac12,\quad v=b+\frac12,\quad
 \epsilon_d={\bf1}_{d>h},\quad \epsilon_e={\bf1}_{e>h}.
\]
Stripping the unique possible nonunit factor before taking fractional
parts gives the exact identity
\begin{equation}\label{eq:gamma-stripping}
 P_j=p^{\epsilon_d+\epsilon_e}u^{\epsilon_d}v^{\epsilon_e}
 \frac{\Gamma(u)\Gamma(v)}{c^2\Gamma(1+a)\Gamma(1+b)}.
\end{equation}
Here
$a\equiv e$, $b\equiv d\pmod p$,
$u\in p\Zp$ exactly when $e=h$, and
$v\in p\Zp$ exactly when $d=h$.
For $0<j<N$, $j\ne H$, write
\[
 x=\frac jN,\quad y=\frac{j'}N,\quad
 L_x={\bf1}_{x>1/2},\quad L_y={\bf1}_{y>1/2},\quad
 \kappa=L_x+L_y,
\]
and
\[
 F_j=\frac{\Gamma(\langle1/2-x\rangle)
                 \Gamma(\langle1/2-y\rangle)}
 {c^2\Gamma(1-x)\Gamma(1-y)},
 \qquad \langle t\rangle=t-\lfloor t\rfloor.
\]
Applying the Morita functional equation only to negative arguments in
\eqref{eq:gamma-stripping}, the four exhaustive carry cases are
\[
\begin{array}{c|c|c|c}
\text{digit condition}&(L_x,L_y)&\text{remaining factor}&P_j\\ \hline
 d\ne h,\ e\ne h&(\epsilon_d,\epsilon_e)&1&(-p)^\kappa F_j\\
 d=h,\ e<h\ \text{or}\ e=h,\ d<h&(0,0)&1&F_j\\
 d=h,\ e>h&(1,1)&v/u=p&p^2F_j\\
 e=h,\ d>h&(1,1)&u/v=p&p^2F_j
\end{array}
\]
Thus $P_j=(-p)^\kappa F_j$ throughout $0<j<N$ away from $H$.
At the omitted middle index $j=H$ we have $a=b=-1/2$; using
\eqref{eq:gamma-stripping} before introducing negative powers of $u,v$
gives
\begin{equation}\label{eq:three-exceptional-P}
 P_H=c^{-4}=1,
 \qquad P_0=1,
 \qquad P_N=[p(-1/2)U(-1)]^2=q.
\end{equation}
Here $c^2=(-1)^{h+1}=1$ and $U(-1)=-2$.

Let $\omega$ be the Teichm\"uller character of $\Fq$, let
$\phi=\omega^H$, and choose $\pi^{p-1}=-p$ compatibly with the additive
character.  For $0\le k<N$, Gross--Koblitz in our convention reads
\begin{equation}\label{eq:GK-exact}
 g(\omega^{-k})=-\pi^{s_p(k)}
 \Gamma_p(k/N)\Gamma_p(\langle pk/N\rangle).
\end{equation}
For $k\ne0$ this is exactly Gross--Koblitz
\cite[Thm.~1.7, pp.~570--571]{GrossKoblitz}, after accounting for the
minus sign built into their Gauss-sum convention.  The case $k=0$ is
not obtained by extending that theorem beyond its stated range: it is
checked directly from $g(\varepsilon)=-1$ and $\Gamma_p(0)=1$.
In particular $g(\phi)=p$.

For $0<j<N$, put $k=N-j$ and let $t$ represent $H-j\pmod N$ in
$\{0,\ldots,N-1\}$.  Rotation of the two base-$p$ digits gives
\[
 \frac{s_p(k)}{p-1}=2-x-y,
 \qquad
 \frac{s_p(t)}{p-1}=1-x-y+\kappa.
\]
Applying \eqref{eq:GK-exact} to $k,t,H$ therefore yields
\[
 F_j=-\pi^{s_p(k)+(p-1)-s_p(t)}
 \frac{g(\phi\omega^j)}{g(\phi)g(\omega^j)}.
\]
The exponent is $(p-1)(2-\kappa)$; multiplication by
$(-p)^\kappa$ gives
\begin{equation}\label{eq:P-gauss-ratio}
 P_j=-q\frac{g(\phi\omega^j)}{g(\phi)g(\omega^j)}
 \qquad(0<j<N).
\end{equation}
For $\omega^j\ne\varepsilon,\phi$, the Gauss--Jacobi identities and
$\phi(-1)=1$ convert \eqref{eq:P-gauss-ratio} to
$P_j=-J(\phi,\omega^{-j})$.  At the two exceptional characters,
$J(\phi,\varepsilon)=J(\phi,\phi)=-1$, so the direct values
$P_0=P_H=1$ give the same formula.  The extra digit endpoint $j=N$
remains separate and contributes $P_N=q$.  We have proved
\begin{equation}\label{eq:sigma-jacobi}
 P_j=-J(\phi,\omega^{-j})\quad(0\le j<N),
 \qquad P_N=q,
 \qquad
 \Sigma_{p,2}=q^3-\sum_\chi J(\phi,\chi)^3.
\end{equation}

For completeness, we evaluate the last moment with its exceptional terms
visible.  Greene's conventions are
\[
 \binom AB=\frac{B(-1)}qJ(A,\bar B),
 \qquad
 \mathcal F_q:={}_3F_2\!\left(
 \begin{matrix}\phi,&\phi,&\phi\\&\varepsilon,&\varepsilon\end{matrix}
 \middle|1\right)_q
 =\frac q{q-1}\sum_\chi\binom{\phi\chi}{\chi}^{\!3}.
\]
A M\"obius substitution gives
$\binom{\phi\chi}{\chi}=q^{-1}J(\phi,\bar\chi)$ for every $\chi$,
including $\varepsilon$ and $\phi$.  Hence
\begin{equation}\label{eq:jacobi-cube-hypergeom}
 \sum_\chi J(\phi,\chi)^3=q^2(q-1)\mathcal F_q.
\end{equation}
These are Greene's Definition~2.4, identity~(2.6), and
Definition~3.10 \cite[pp.~80, 83]{Greene}.

Because $q\equiv1\pmod4$, choose a quartic character $\rho$ with
$\rho^2=\phi$.  Specializing Greene's finite-field Dixon theorem
\cite[Thm.~4.37, p.~97]{Greene} with $A=B=C=\phi$ and $D=\rho$,
the three correction terms contain respectively
$\delta(A)$, $\delta(B)$, and $\delta(C\overline{AB})$.
All three arguments equal the nontrivial character $\phi$, so every
correction term is zero.  The two remaining products reduce, using
Greene's identity~(2.6), to
\[
 \mathcal F_q
 =\binom{\rho}{\phi}^{\!2}
  +\binom{\bar\rho}{\phi}^{\!2}
 =\frac{J(\rho,\phi)^2+J(\bar\rho,\phi)^2}{q^2}.
\]
Since $p\equiv-1\pmod4$, Frobenius interchanges $\rho$ and $\bar\rho$;
the change of variable $x\mapsto x^p$ therefore gives
$J(\rho,\phi)=J(\bar\rho,\phi)$.  The two sums are complex conjugates,
and the standard Jacobi norm identity gives their product $q$.
Consequently each square is $q$, so
\[
 \mathcal F_q=\frac2q,
 \qquad
 \sum_\chi J(\phi,\chi)^3=2q(q-1).
\]
Together with \eqref{eq:sigma-jacobi}, this proves
\[
 \Sigma_{p,2}=q^3-2q(q-1)=p^6+2p^2(1-p^2).
\]
Finally, \eqref{eq:trace-quotient}, \eqref{eq:K-fixed-trace}, and
\eqref{eq:T2-fixed-trace} give
\[
 \operatorname{tr}(M_{p,\infty}^2)
 =\frac{\Sigma_{p,2}-p^6}{1-p^2}=2p^2.
\]
Since the quotient is two-dimensional and its first trace is zero,
$\det M_{p,\infty}=-\tfrac12\operatorname{tr}(M_{p,\infty}^2)=-p^2$.
Its characteristic polynomial is therefore $X^2-p^2$, and
Cayley--Hamilton yields $M_{p,\infty}^2=p^2I$.  Applying this identity
to the class of $1$ proves \eqref{eq:zeroquotient}.
\end{proof}

\begin{remark}
At $p=3$ the second-trace identity can be checked without invoking the
general finite-field Dixon formula.  In $\mathbb F_9=\mathbb F_3(i)$,
the eight Jacobi sums have cubes summing to $144$.  Hence
$\Sigma_{3,2}=729-144=585$ and
$\operatorname{tr}(M_{3,\infty}^2)=(585-729)/(1-9)=18$.
\end{remark}

\section{From zero quotient to an actual primitive}\label{sec:primitive}

Write $F_n=a\fall n$ and $P_k=(a-1)\fall k$.  A direct calculation gives
\begin{equation}\label{eq:three-term}
 DP_k=-A_k^*F_k-B_k^*F_{k+1}-C_k^*F_{k+2},
\end{equation}
where
\[
 A_k^*=\frac{(2k+1)^3}{8},\quad
 B_k^*=\frac{8(k+1)^2+1}{4},\quad
 C_k^*=\frac{2k+3}{2}.
\]
Let $G_n=n!g_n$, $H_n=n!b_n$.  They satisfy the homogeneous recurrence
associated with \eqref{eq:three-term}, and their Wronskian is
\begin{equation}\label{eq:wronskian}
 W_n=G_nH_{n+1}-G_{n+1}H_n
 =\frac{(n!)^2\binom{2n}{n}^2}{16^n(2n+1)}.
\end{equation}
For $0\le k\le n-2$, set
\[
 c_{n,k}=-\frac{G_{k+1}H_n-H_{k+1}G_n}{C_k^*W_{k+1}}.
\]
Then
\begin{equation}\label{eq:green}
 F_n=D\left(\sum_{k=0}^{n-2}c_{n,k}P_k\right)+G_n+H_nu.
\end{equation}
The index $k+1$ in the Wronskian is forced by the inhomogeneous initial
condition at $n=k+2$.

\begin{lemma}[Logarithmic-loss inverse]\label{lem:green-bound}
For each $n\ge2$ there are unique polynomials $Q_n,R_n$ such that
\[
 a^n=DQ_n+R_n,\qquad \deg R_n\le1,
\]
and, for every odd prime $p$,
\[
 \|Q_n\|\le p^{6\lfloor\log_p(2n+1)\rfloor}.
\]
If $f=\sum t_na^n$ has $\vp(t_n)\ge\sigma n-C$ with $\sigma>0$ and
$\Pi f=0$, then $q_f=\sum_{n\ge2}t_nQ_n$ converges on every disc of
radius $|a|_p\le p^\delta$ with $0<\delta<\sigma$, and satisfies
$Dq_f=f$.
\end{lemma}
\begin{proof}
The full Green-kernel calculation is in
Appendix~\ref{app:green-denominators}.  Briefly, the factor $2m+1$ in
$C_k^*$ cancels against the denominator of the Casoratian
\eqref{eq:wronskian}, leaving
\[
 c_{n,k}=-\frac{2\cdot16^mn!}{m!\binom{2m}{m}^2}
 (g_mb_n-b_mg_n),\qquad m=k+1.
\]
The moment bound and Kummer's carry formula give the stated logarithmic
loss.  Hence the primitive series converges on the indicated discs.
Termwise polynomial division leaves a convergent linear remainder, and that
remainder is exactly the continuous projection $\Pi f$; it is zero by
hypothesis.  This constructs an actual primitive and does not assume that
$D\Aalg$ is closed.
\end{proof}

For $p\ge7$, the remaining integral estimate follows from finite-product
Dixon arithmetic.  We record the exact form needed here.

\begin{lemma}[Low-degree arithmetic]\label{lem:large-prime-arithmetic}
Let $p\equiv3\pmod4$ and $p\ge7$, and write
$E_0=(T^2-p^2)1=\sum t_na^n$.  Then $S\in p^2\OO$ and
$E_0(a)\in p^5\Zp$ for $a=0,1,\ldots,5$.  Consequently
$E_0\in p^5\OO$.
\end{lemma}
\begin{proof}
The finite-product calculation is given in
Appendix~\ref{app:large-prime}.  Its two outputs are the pointwise congruence
$T^21(a)\equiv p^2\pmod {p^5}$ for every nonnegative integer $a$, and
$S(0),S(1)\in p^2\Zp$.  The calculation uses classical terminating Dixon
and differentiated odd-Dixon identities; \cite{Tauraso} is cited only for
those one-block identities, not for the present two-block congruence.

We record here the analytic upgrade, since it is logically independent of
the finite calculation.  By \cref{prop:bidisc},
$\vp(t_n)\ge\alpha_pn$, and $6\alpha_p>4$ for $p\ge7$; hence
$t_n\in p^5\Zp$ for $n\ge6$.  Subtracting this tail shows that the
degree-five truncation takes values in $p^5\Zp$ at $0,\ldots,5$.  Its
Vandermonde determinant is $\prod_{j=1}^{5}j!$, a $p$-adic unit, so its
six coefficients lie in $p^5\Zp$.  Thus $E_0\in p^5\OO$.
The same argument with two nodes and $2\alpha_p>1$ gives $S\in p^2\OO$.
\end{proof}

\begin{lemma}[Triangular denominator budget]\label{lem:double-factorial}
For the polynomials $Q_n,R_n$ determined by
\[
 a^n=DQ_n+R_n,\qquad \deg R_n\le1,
\]
one has
\[
 \deg Q_n\le n-2,\qquad
 \|Q_n\|,\|R_n\|\le p^{B_p(n)},\qquad
 B_p(n)=\vp((2n-1)!! )\le\frac{2n}{p-1}.
\]
\end{lemma}
\begin{proof}
This is the descending elimination proved in
Appendix~\ref{app:green-denominators}.  The important point is that every
nonunit pivot $(2j+3)/2$ is paid for at most once; all other coefficients of
$D(a^j)$ are $p$-integral.
\end{proof}

Combining \cref{prop:spectrum,lem:green-bound,lem:large-prime-arithmetic}
gives $Dq=E_0$.  Descending polynomial division gives the sharper bound
$\|Q_n\|\le p^{v_p((2n-1)!!)}\le p^{2n/(p-1)}$ by
\cref{lem:double-factorial}.  It follows termwise that
$q\in p^5\OO$ for $p\ge19$, while $q\in p^4\OO$ for $p=7,11$.
Moreover $\alpha_p>1/2$ for $p\ge7$, so the first parameter coefficient
satisfies $\Lambda\OO\subseteq\LL$.  This proves
\cref{prop:pure-input} away from $p=3$.

\section{The exceptional prime three}\label{sec:p3}

At $p=3$, $\alpha_3=1/6$.  This remains enough for convergence but not for
the preceding quick first-jet estimate.  Pairing the retained factors
$6b+1$ and $6b+5$ gives
\begin{equation}\label{eq:p3-log}
 \frac{[N]V_N(a)}{V(a)}
 =-\frac65\sum_{k\ge0}\left(-\frac{36}{5}\right)^kP_k(a),\qquad
 P_k(a)=\sum_{b=0}^{a-1}(2b+1)[b(b+1)]^k.
\end{equation}
The polynomial $P_k$ has degree at most $2k+2$, vanishes at zero, and its
Taylor denominators divide $(2k+2)!$.  Since
$v_3((2k+2)!)\le k$, the series \eqref{eq:p3-log} converges in
$3a\OO$.  Differentiating the three remaining digit factors yields
\[
 Tf\equiv f(0)-f(1),\qquad \Lambda f\equiv f(1)\pmod{3\OO}.
\]
Thus $T\OO,\Lambda\OO\subseteq\LL$.  A direct calculation gives
$S=VQ$, where, on putting $x=3a$,
\[
 Q(a)=1+\left(\frac{x+1/2}{x+1}\right)^3+
 \left\{\frac{(x+1/2)(x+3/2)}{(x+1)(x+2)}\right\}^{\!3}.
\]
Thus $Q$ is a rational function of $x=3a$ whose denominator is a unit
in $\mathbb Z_3\langle a\rangle$.  Consequently the coefficient of $a^n$ in
$Q$ is divisible by $3^n$.  Direct evaluation gives
\[
 Q(0)=\frac{603}{512}=9\frac{67}{512},\qquad
 Q'(0)=\frac{1719}{1024}=9\frac{191}{1024}.
\]
The coefficients of degrees $n\ge2$ are already divisible by $3^n$, hence
by $9$; the two displayed values deal with degrees zero and one.  Therefore
$Q\in9\OO$.  Since $V\in1+3a\OO$, it follows that $S=VQ\in9\OO$.

It remains to prove that the primitive of $E_0=(T^2-9)1$ lies in
$81\OO$.  Write $E_0=\sum t_na^n$.  The gamma bound and
\cref{lem:green-bound} show that for $n\ge199$,
\[
 \vp(t_nQ_n)\ge
 \left\lceil\frac n6\right\rceil
 -6\lfloor\log_3(2n+1)\rfloor\ge4.
\]
Thus only $2\le n\le198$ require finite certification.

\begin{proposition}[Finite $3$-adic certificate]\label{prop:p3-cert}
For $2\le n\le198$, all coefficients of
$\sum t_nQ_n$ lie in $81\mathbb Z_3$.  The assertion is certified by exact
rational arithmetic followed by reduction modulo $3^{34}$; a second run
modulo $3^{38}$ gives the same normalized vector modulo $3^6$.
\end{proposition}
\begin{proof}[Certificate specification]
Appendix~\ref{app:p3} specifies the exact gamma truncation, its propagated
$3^P$ error ball, the rational polynomial divisions and the coefficientwise
test.  In particular, the program checks the finite Taylor range
$2\le n\le198$; it does not test the original congruence at finitely many
values.  The vanishing of the accumulated linear remainder comes from
\cref{prop:spectrum}, not from the program.
\end{proof}

Together with the exact quotient vanishing \eqref{eq:zeroquotient}, the
finite certificate and the infinite tail give an actual
$q\in81\OO$ with $(T^2-9)1=Dq$.  This completes
\cref{prop:pure-input} at $p=3$.

\section{All-depth descent and proof of the theorem}

We now use only \cref{lem:endpoints,prop:bidisc,prop:pure-input}.
Every positive total-degree coefficient in \cref{prop:bidisc} is divisible
by $p$, and $f(pa+d)\equiv f(d)\pmod{p\OO}$.  Hence
$T_N\OO\subseteq\LL$ for $N\in p\Zp$.  Since
$T(c+pf)=cS+pTf$, we also have $T\LL\subseteq p\LL$.

Define
\begin{equation}\label{eq:YX}
 Y=\Lambda S+pT\Lambda1,\qquad
 X=Y-\frac{B_2q}{2p}.
\end{equation}
Then $Y,X\in p^2\LL$.  In bounded-operator norm write
\[
 T_{pM}=\sum_{n\ge0}M^nC_n,\qquad
 T_{p^2M}=\sum_{n\ge0}p^nM^nC_n,\qquad
 C_0=T,\ C_1=p\Lambda.
\]
If $b_n=\lceil\alpha_pn\rceil$, then $C_n\OO\subseteq p^{b_n}\OO$.
The coefficient of $M^n$ in $T_{pM}T_{p^2M}1$ is
$\sum_{i+j=n}p^jC_iC_j1$.  Its constant and linear parts are $T^21$ and
$pY$.  For $n\ge2$, the two edge terms have valuations at least
$b_n+2$ and $n+b_n$, while each mixed term has valuation at least
\[
 j+b_{n-j}+b_j\ge\max\{3,1+b_n\}.
\]
The bounds tend to infinity, even when $p=3$.  Therefore
\begin{equation}\label{eq:master}
 (T_{pM}T_{p^2M}-p^2)1
 =D_Mq+pMX+p^3M^2R(a,M),\qquad
 R\in\Zp\langle a,M\rangle.
\end{equation}

The same expansion gives the two lattice laws
\begin{equation}\label{eq:lattice-laws}
 T_N\LL\subseteq p\LL\quad(N\in p^2\Zp),\qquad
 T_N\LL\subseteq p\OO\quad(N\in p\Zp).
\end{equation}
Write $M=p^em$ with $p\nmid m$.  Iterating \cref{lem:endpoints}, with the
rightmost operator acting first, gives
\[
 \JJ_M=\JJ_mT_{pm}T_{p^2m}\cdots T_{p^em}.
\]
Starting from $X\in p^2\LL$, the first $e-1$ steps in
\eqref{eq:lattice-laws} gain one power and retain $\LL$; the last gains
one more and lands in $\OO$.  Hence
\begin{equation}\label{eq:JX}
 \JJ_M(X)\in p^{e+2}\Zp.
\end{equation}
For $e=0$, this follows directly from \eqref{eq:Jintegral}.

Apply $\JJ_M$ to \eqref{eq:master}.  The coboundary is annihilated by
\cref{lem:endpoints}, and \eqref{eq:JX} gives
\begin{equation}\label{eq:normalized}
 \TT(p^2M)-p^2\TT(M)\in p^{2e+3}\Zp.
\end{equation}
Starting the same descent from $1\in\LL$ also gives
$H(M)=\TT(M)4^{M-1}\in p^e\Zp$.  Finally,
\[
 H(p^2M)-p^2H(M)
 =4^{p^2M-1}\{\TT(p^2M)-p^2\TT(M)\}
 +p^2\{4^{(p^2-1)M}-1\}H(M).
\]
By Fermat and LTE,
$\vp(4^{(p^2-1)M}-1)\ge e+1$.  Both terms lie in
$p^{2e+3}\Zp$, proving \eqref{eq:strong-main}.  Taking
$M=mp^{r-2}$ proves \cref{thm:main}.

\section{Reproducibility}

The supplementary PARI/GP certificate program implements
\cref{prop:p3-cert}.  Its full filename is recorded in the accompanying
README.  It was run with PARI/GP 2.15.5 and a
256 MB stack.  The wrapper fixes the stack size, captures both output
streams, and fails on PARI error markers rather than trusting an exit code
or a printed \texttt{PASS}.  Saved runs at precisions $3^{34}$ and $3^{38}$
agree as described above.  No finite test of the original congruence is used
in the proof.

\appendix
\section{The large-prime finite-product calculation}
\label{app:large-prime}

This appendix supplies the calculation used in
\cref{lem:large-prime-arithmetic}.  Throughout,
$p=4m+3\ge7$, $h=(p-1)/2=2m+1$, and
\[
 A_n=\frac{(1/2)_n^3}{(n!)^3}
     =\frac{\binom{2n}{n}^3}{64^n}.
\]
Write $S=T1$, $L=S/p^2$ and $E_0=(T^2-p^2)1=p^2(TL-1)$.
No high-order Newton filtration is used below.

\subsection{A fixed-modulus representative for one digit block}

For nonnegative integers $a$, put
\[
 V(a)=\frac{A_{pa}}{A_a},\qquad
 U(a)=\sum_{d=0}^{p-1}
 \left(\frac{(-h+p/2+pa)_d}{(1+pa)_d}\right)^3.
\]
Then $S(a)=V(a)U(a)$.  Removing the factors divisible by $p$ from the
finite products gives
\begin{equation}\label{eq:app-logV}
 \log V(a)=-a\log(64^{p-1})+
 3\sum_{r\ge1}\frac{(-1)^{r+1}p^r}{r}
 H_{p-1}^{(r)}P_r(a),
 \quad
 P_r(a)=\sum_{b=0}^{2a-1}b^r-2\sum_{b=0}^{a-1}b^r.
\end{equation}
This identity is valid for every integer $a\ge0$; no restriction such as
$a<p$ is imposed.  For $r\ge5$, $p^r/r\in p^5\Zp$; for $r=4$ use
$H_{p-1}^{(4)}\in p\Zp$; and for $r=1,2,3$ use
\[
 H_{p-1}\in p^2\Zp,\qquad
 H_{p-1}^{(2)},H_{p-1}^{(3)}\in p\Zp.
\]
Since $P_1,P_2,P_3$ have degrees $2,3,4$, respectively, exponentiating
\eqref{eq:app-logV} proves that there is a polynomial $V_*$ with
\begin{equation}\label{eq:app-Vstar}
 V(a)\equiv V_*(a)\pmod {p^5},\quad
 \deg V_*\le4,\quad [a^j]V_*\in p^j\Zp\ (1\le j\le4).
\end{equation}

We need only the following exact unweighted finite-field row identity.  If
$g_d(x)=((1/2+x)_d/(1+x)_d)^3$, then
\begin{equation}\label{eq:app-row-zero}
 \sum_{d=0}^{p-1}g_d(x)=0\quad\text{in }\mathbb F_p(x).
\end{equation}
Indeed, $1/2=-h=h+1$ in $\mathbb F_p$ and
\[
 \frac{(z+h)_d}{(z)_d}=\frac{(z+d)_h}{(z)_h}
\]
shows that $(x+1)_h^3g_d(x)=P(x+d+1)$ with $P(y)=(y)_h^3$.
Write $P(y)=\sum_{k=0}^{3h}a_ky^k$.  Since
\[
 \sum_{c\in\mathbb F_p}c^i=
 \begin{cases}-1,&i>0\text{ and }p-1\mid i,\\0,&\text{otherwise},
 \end{cases}
\]
and $3h<2(p-1)$, the only potentially nonzero term in
$\sum_cP(x+c)$ comes from $i=p-1$ in the binomial expansion of
$(x+c)^k$.  For $p\le k\le3h<2(p-1)$, Lucas' theorem gives
$\binom{k}{p-1}\equiv0\pmod p$.  The sole remaining possibility
$k=p-1$ is a constant in $x$.  Evaluating the sum at $x=0$ shows that
this constant equals
$(h!)^3\sum_{d=0}^{h}A_d=0$ by the terminating odd Dixon identity.
Hence $\sum_{c\in\mathbb F_p}P(x+c)=0$.  Dividing by $(x+1)_h^3$ proves
\eqref{eq:app-row-zero}.  No weighted identity with an arbitrary
polynomial factor is asserted or used.
Expanding the finite two-variable row at $x=pa$ and at the parameter
shift $p/2$, equation \eqref{eq:app-row-zero} gives
\begin{equation}\label{eq:app-Ubound}
 v_p([a^j]U)\ge j+1\qquad(j\ge0).
\end{equation}
Indeed, before substitution this is a finite rational function whose
denominators are units at $(x,\varepsilon)=(0,0)$.  Its constant term
vanishes modulo $p$ by \eqref{eq:app-row-zero}; after
$(x,\varepsilon)=(pa,p/2)$ every occurrence of $a^j$ contributes $p^j$
and the vanished constant supplies the additional factor $p$.
The constant term gains one further power.  Put
\[
 \mathcal P(z)=\sum_{d=0}^{h}\frac{(z)_d^3}{(d!)^3},\qquad
 c_m=\frac{(-1)^m(6m+3)!(m!)^3}
 {(3m+1)!((2m+1)!)^3}.
\]
Odd Dixon summation and its first derivative give
\[
 \mathcal P(-h)=0,\qquad \mathcal P'(-h)=c_m/2,\qquad \vp(c_m)=1.
\]
The terms $d>h$ are in $p^3\Zp$, so Taylor expansion at $-h$, evaluated
at $-h+p/2$, gives $U(0)\in p^2\Zp$.  Combining
\eqref{eq:app-Vstar} with \eqref{eq:app-Ubound} yields
\begin{equation}\label{eq:app-Sshape}
 S(a)\equiv s_0+s_1a+s_2a^2+s_3a^3\pmod {p^5},\qquad
 \vp(s_0),\vp(s_1)\ge2,\quad\vp(s_2)\ge3,\quad\vp(s_3)\ge4.
\end{equation}
Consequently, in the Newton basis,
\begin{equation}\label{eq:app-Lshape}
 L(n)\equiv c_0+c_1n+pb_2\binom n2+c_3\binom n3\pmod {p^3},
 \qquad c_3\in p^2\Zp.
\end{equation}
Only degrees at most three occur, so the denominators $2!$ and $3!$ are
units.  This is the promised fixed-degree replacement for the unused
Newton ladder.

\subsection{The one-digit moments}

For $0\le d\le h$, set
\[
 D_d=H_{2d}-H_d,\quad Q=q_p(64),\quad E_d=6D_d-Q,
 \quad F_d=3(H_d^{(2)}-2H_{2d}^{(2)})+\frac{E_d^2}{2}.
\]
A direct finite-product expansion gives
\begin{equation}\label{eq:app-Kexp}
 \frac{A_{pa+d}}{A_a}\equiv A_d
 \left\{1+paE_d+p^2\left(\frac{aQ^2}{2}+a^2F_d\right)\right\}
 \pmod {p^3}.
\end{equation}
For $d>h$ the left side is in $p^3\Zp$.  Substitution of
\eqref{eq:app-Lshape} and \eqref{eq:app-Kexp} into $TL$ gives
\begin{equation}\label{eq:app-TLshape}
 (TL)(a)-(TL)(0)\equiv paX+
 p^2\{a(b_2M+Q^2/2)+a^2Z\}\pmod {p^3},
\end{equation}
where $M=\sum_{d=0}^{h}dA_d$ and
$X=\sum_{d=0}^{h}A_dE_dL(d)$.  More precisely,
\begin{equation}\label{eq:app-XYZ}
 \begin{aligned}
 X&=\sum_{d=0}^{h}A_dE_dL(d),\\
 Z&\equiv\sum_{d=0}^{h}A_d\{c_1E_d+F_dL(d)\}\pmod p.
 \end{aligned}
\end{equation}

To evaluate them, put
\[
 w_d=(-1)^d\binom hd^3,\qquad B_d=H_{h-d}-H_d.
\]
Modulo $p$ one has
\[
 A_d=w_d,\qquad E_d=3B_d,\qquad
 F_d=\widehat F_d:=\frac32(H_d^{(2)}+H_{h-d}^{(2)})+\frac92B_d^2.
\]
The differentiated odd-Dixon identity
\begin{equation}\label{eq:app-tauraso-input}
 \sum_{d=0}^{h}w_dH_d=-\frac{c_m}{6}
\end{equation}
and reflection $d\mapsto h-d$ show that the first two moments of $E_d$
vanish modulo $p$.  For the second moments, take residues of
\[
 R_h(x)=\frac{(h!)^3}{(x)_{h+1}^3}.
\]
At $x=-d$ its $y^{-1}$ coefficient, with $y=x+d$, is
$w_d\widehat F_d$.  The sums of residues of $R_h$ and $xR_h$ are zero,
and hence
\begin{equation}\label{eq:app-residues}
 \sum_{d=0}^{h}w_d\widehat F_d=0,\qquad
 \sum_{d=0}^{h}w_d(2d-h)\widehat F_d=-2c_m.
\end{equation}
Since $p\mid c_m$, these identities give $Z=0$ in
\eqref{eq:app-TLshape}.  To make the cancellation explicit, write
$L(d)\equiv u+\nu d\pmod p$.  Then
\[
 Z\equiv
 c_1\sum_{d=0}^{h}A_dE_d
 +u\sum_{d=0}^{h}A_dF_d
 +\nu\sum_{d=0}^{h}dA_dF_d\pmod p.
\]
The first sum is zero by reflection and
$\sum w_dB_d=c_m/3\in p\Zp$.  The first identity in
\eqref{eq:app-residues} kills the second sum.  The second identity,
together with $p\mid c_m$, kills the third.  Thus every summand in the
displayed expression for $Z$ vanishes separately; in particular
\begin{equation}\label{eq:app-Zzero}
 Z=0\qquad\text{in }\mathbb F_p.
\end{equation}

The Newton-curvature term is removed by the finite shift identity
\[
 \sum_{d=0}^{h}w_d\{d^2+d^3B_d\}=0,\qquad
 \sum_{d=0}^{h}w_d\{2d+3d^2B_d\}=h^2c_m/6.
\]
Both follow by shifting the index in the first sum and using
$w_{d+1}/w_d=-(h-d)^3/(d+1)^3$ and
$B_{d+1}=B_d-(h-d)^{-1}-(d+1)^{-1}$.  Consequently
\begin{equation}\label{eq:app-curvature}
 \sum_{d=0}^{h}\binom d2A_dE_d=-M\pmod p,
\end{equation}
so the $b_2$ contribution to $X/p$ cancels the term $b_2M$ in
\eqref{eq:app-TLshape}.

It remains to evaluate the affine part of $L$.  Put
\[
 S_0=\sum_{d=0}^{h}A_d,\qquad S_D=\sum_{d=0}^{h}A_dD_d,
 \qquad \sigma_0=p^{-2}S_0\pmod p,\qquad
 \sigma_1=p^{-1}S_D\pmod p.
\]
For \(0\le d\le h\), regrouping the next digit gives
\[
 S(d)=\sum_{e=0}^{p-1}\frac{A_{pd+e}}{A_d}.
\]
Reducing \eqref{eq:app-Kexp} modulo \(p^2\), the terms \(e>h\) vanish
and therefore
\[
 S(d)\equiv S_0+pd\sum_{e=0}^{h}A_eE_e\pmod {p^3}.
\]
Since
\[
 \sum_{e=0}^{h}A_eE_e=6S_D-QS_0
 \equiv6p\sigma_1\pmod {p^2},
\]
division by \(p^2\) gives the first-order affine row
\begin{equation}\label{eq:app-affine-pre}
 L(d)\equiv\sigma_0+6d\sigma_1\pmod p.
\end{equation}
We now derive the relation between its two coefficients.  Put
\[
 a_0=-h,\qquad
 \ell_d=\sum_{j=0}^{d-1}\frac1{a_0+j},\qquad
 q_d=\sum_{j=0}^{d-1}\frac1{(a_0+j)^2}.
\]
Taylor expansion of $\mathcal P(a)$ at $a_0$, using
$\mathcal P(a_0)=0$ and $\mathcal P'(a_0)=c_m/2$, gives
\begin{equation}\label{eq:app-sigma0-explicit}
 \sigma_0\equiv\frac{c_m}{4p}
 +\frac98\sum_{d=0}^{h}w_d\ell_d^2
 -\frac38\sum_{d=0}^{h}w_dq_d\pmod p.
\end{equation}
Let
\[
 O_d=\sum_{j=1}^{d}\frac1{2j-1},\qquad
 O_d^{(2)}=\sum_{j=1}^{d}\frac1{(2j-1)^2}.
\]
Since $p=2h+1$,
\[
 \ell_d=-2\sum_{j=1}^{d}\frac1{p-(2j-1)}
 \equiv2O_d+2pO_d^{(2)}\pmod {p^2},
\]
so that
\[
 D_d\equiv\frac{\ell_d-H_d}{2}-pO_d^{(2)}\pmod {p^2},
 \qquad q_d\equiv4O_d^{(2)}\pmod p.
\]
Also
\[
 A_d\equiv w_d\left(1+\frac{3p}{2}\ell_d\right)\pmod {p^2},
 \qquad\sum_{d=0}^{h}w_d\ell_d=\frac{c_m}{6}.
\]
It follows that
\begin{equation}\label{eq:app-sigma1-explicit}
 \sigma_1\equiv\frac{c_m}{6p}
 -\frac14\sum_{d=0}^{h}w_dq_d
 +\frac34\sum_{d=0}^{h}w_d(\ell_d^2-\ell_dH_d)\pmod p.
\end{equation}
The mixed sum is zero modulo $p$.  Indeed,
$\ell_d=-H_h+H_{h-d}$, and hence
\[
 \sum_{d=0}^{h}w_d\ell_dH_d
 =-H_h\sum_{d=0}^{h}w_dH_d
  +\sum_{d=0}^{h}w_dH_{h-d}H_d\equiv0\pmod p.
\]
The first term vanishes by \eqref{eq:app-tauraso-input} and
$p\mid c_m$, and the second by reflection.  Comparing
\eqref{eq:app-sigma0-explicit} and \eqref{eq:app-sigma1-explicit}
gives $\sigma_0\equiv3\sigma_1/2\pmod p$.  Thus
\eqref{eq:app-affine-pre} becomes
\begin{equation}\label{eq:app-affine}
 L(d)\equiv\sigma_0(1+4d)\pmod p.
\end{equation}
This calculation uses only the finite half-block and no information
about $E_0$.

We next lift the weighted moment by one $p$-adic digit.  Finite-product
expansion and harmonic reflection give
\[
 A_d\equiv w_d(1+3pO_d)\pmod {p^2},\qquad
 B_d\equiv H_h+2D_d+2pO_d^{(2)}\pmod {p^2}.
\]
The half-harmonic congruence is
\begin{equation}\label{eq:app-half-harmonic}
 (Q+3H_h)/p\equiv\frac92H_h^2\pmod p.
\end{equation}
Consequently
\begin{equation}\label{eq:app-AE-lift}
 A_dE_d\equiv w_d\left[
 3B_d+p\left\{9O_dB_d-\frac92H_h^2-6O_d^{(2)}\right\}
 \right]\pmod {p^2}.
\end{equation}
The leading weighted sum satisfies
\[
 3\sum_{d=0}^{h}(1+4d)w_dB_d=(2h+1)c_m=pc_m,
\]
and therefore vanishes modulo $p^2$.  Since $w_d(1+4d)$ is invariant
under $d\mapsto h-d$, only the symmetric part of the braces contributes;
it is $\widehat F_d/2-9H_h^2/2$.  Furthermore
$1+4d\equiv2(2d-h)\pmod p$, so \eqref{eq:app-residues} gives
\[
 \sum_{d=0}^{h}(1+4d)w_d\widehat F_d
 \equiv-4c_m\equiv0\pmod p.
\]
Because $\sum w_d=0$ and $\sum dw_d\equiv M\pmod p$,
$\sum(1+4d)w_d\equiv4M\pmod p$.  Dividing
\eqref{eq:app-AE-lift} by $p$, summing, and using
$Q\equiv-3H_h\pmod p$ proves
\begin{equation}\label{eq:app-lifted-moment}
 \frac1p\sum_{d=0}^{h}A_d(1+4d)E_d
 \equiv-18H_h^2M=-2Q^2M\pmod p.
\end{equation}
Finally, $(TL)(0)=1\pmod p$ and \eqref{eq:app-affine} give
$4\sigma_0M=1\pmod p$.  Hence the remaining linear coefficient in
\eqref{eq:app-TLshape} is
\[
 -2\sigma_0Q^2M+Q^2/2=0\pmod p.
\]
We have proved
\begin{equation}\label{eq:app-TLconstant}
 (TL)(a)=(TL)(0)\pmod {p^3}\qquad(a\in\mathbb Z_{\ge0}).
\end{equation}

\subsection{The origin and the Vandermonde upgrade}

Regrouping the two digit blocks at the origin gives
\[
 p^2(TL)(0)=\sum_{k=0}^{p^2-1}A_k.
\]
We recall the finite Dixon calculation at the precision required here.
Put $s=(p^2-1)/2$, $n=s/2$, $z_0=-s$, and $\delta=p^2/2$.  For
$\mathcal D(z)=\sum_{k=0}^{s}(z)_k^3/(k!)^3$, ordinary Dixon summation
and its first two derivatives give
\[
 \mathcal D(z_0)=\frac{(3n)!}{(n!)^3},\quad
 \frac{\mathcal D'(z_0)}{\mathcal D(z_0)}
 =\frac32(H_n-H_{3n}),\quad
 (\log\mathcal D)''(z_0)=\frac34(H_n^{(2)}-3H_{3n}^{(2)}).
\]
The omitted infinite-series terms have a triple zero at $z_0$, so these
first two derivatives agree with the gamma-quotient form of Dixon's
identity.  The cubic and quartic layers require a summed pole estimate.
Put
\[
 W_k=(-1)^k\binom sk^3,
 \qquad
 C_r=\sum_{k=0}^{s}W_k[z^r]
 \prod_{j=0}^{k-1}\left(1+\frac{z}{z_0+j}\right)^3.
\]
Then $\mathcal D(z_0+\delta)=\sum_{r\ge0}\delta^rC_r$.  Every
denominator $z_0+j$, $0\le j<s$, has valuation at most one, whence
\begin{equation}\label{eq:app-Cr-rough}
 \vp(C_r)\ge-r.
\end{equation}
Thus $\delta^rC_r\in p^5\Zp$ for $r\ge5$.  For $r=3,4$ the gain occurs
only after summation.

The singular denominators are
\[
 z_0+(h+bp)=p(b-h),\qquad0\le b<h.
\]
Write $k=ap+d$, $0\le d<p$, and define
\[
 \tau_j(a)=[z^j]\prod_{b=0}^{a-1}
 \left(1+\frac{z}{b-h}\right)^3.
\]
For $d>h$, Kummer's theorem gives $W_k\in p^3\Zp$; for $d\le h$,
Lucas' theorem gives $W_k\equiv w_aw_d\pmod p$.  Hence
\[
 p^4C_4\equiv
 \sum_{a=0}^{h}w_a\tau_4(a)\sum_{d=0}^{h}w_d=0\pmod p,
\]
so $\vp(C_4)\ge-3$ and $\vp(\delta^4C_4)\ge5$.

For the cubic layer retain one further digit.  Uniformly for
$0\le a,d\le h$,
\[
 W_{ap+d}\equiv w_aw_d\{1+3pL_{a,d}\}\pmod {p^2},
 \qquad
 L_{a,d}=hH_h-aH_d-(h-a)H_{h-d}.
\]
The nonsingular factors have first coefficient
\[
 \nu_d\equiv-3(H_h-H_{h-d})\pmod p.
\]
If $T_h=\sum_{d=0}^{h}w_dH_d$, reflection gives
\[
 \sum_{d=0}^{h}w_dL_{a,d}=(h-2a)T_h,\qquad
 \sum_{d=0}^{h}w_d\nu_d=-3T_h.
\]
Collecting the next-to-leading cubic poles therefore yields
\begin{equation}\label{eq:app-C3-pole}
 p^3C_3\equiv3pT_h\sum_{a=0}^{h}w_a
 \{(h-2a)\tau_3(a)-\tau_2(a)\}\pmod {p^2}.
\end{equation}
By \eqref{eq:app-tauraso-input}, $T_h=-c_m/6\in p\Zp$; hence
$\vp(C_3)\ge-1$ and $\vp(\delta^3C_3)\ge5$.  Together with
\eqref{eq:app-Cr-rough}, this proves
\[
 \mathcal D(z_0+\delta)\equiv
 C_0+\delta C_1+\delta^2C_2\pmod {p^5}.
\]
In particular, the cubic and quartic layers are not discarded term by
term: their extra divisibility is a cancellation in the finite sum.

Writing $p=4m+3$, put $\rho=3m+2$.  Splitting the factorials into
complete $p$-blocks and taking the unit logarithm gives
\[
 \log\left(\frac1{p^2}\frac{(3n)!}{(n!)^3}\right)
 =3pm(H_m-H_\rho)
 +p^2\{-1-\rho^2/2\}H_m^{(2)}
 +\frac32m^2p^2H_\rho^{(2)}\pmod {p^3}.
\]
Using $H_\rho=H_m+pH_m^{(2)}\pmod {p^2}$ and
$H_\rho^{(2)}=-H_m^{(2)}\pmod p$ gives
\[
 \frac1{p^2}\frac{(3n)!}{(n!)^3}
 =1+\frac38p^2H_m^{(2)}\pmod {p^3};
\]
the displayed first and second logarithmic derivatives give the reciprocal
factor
$1-\frac38p^2H_m^{(2)}\pmod {p^3}$.  Thus the front range is
$p^2\pmod {p^5}$.

We give the one-carry calculation for the complementary rear range
\[
 \mathcal R_p=\sum_{k=(p^2+1)/2}^{p^2-1}A_k.
\]
Write $k=ap+d$, $0\le a,d<p$.  Kummer's theorem shows that exactly one
carry occurs precisely when
\[
 h<a<p,\qquad0\le d\le h;
\]
all remaining terms have two carries and lie in $p^6\Zp$.  Therefore
\begin{equation}\label{eq:app-rear-strata}
 \mathcal R_p\equiv
 \sum_{a=h+1}^{p-1}\sum_{d=0}^{h}A_{ap+d}\pmod {p^6}.
\end{equation}
Put
\[
 \Lambda_d=\frac12H_d^{(2)}-H_{2d}^{(2)}.
\]
After removing the common power of $p$, finite-product expansion gives
\[
 \frac{A_{ap+d}}{A_{ap}A_d}
 \equiv\exp\{6paD_d+6p^2a^2\Lambda_d\}\pmod {p^3}.
\]
We have $S_0\in p^2\Zp$, $S_D\in p\Zp$, and reflection gives
\[
 \sum_{d=0}^{h}A_d(6\Lambda_d+18D_d^2)\equiv0\pmod p.
\]
Keeping exactly the relative precision that survives multiplication by
$A_{ap}\in p^3\Zp$ yields
\begin{equation}\label{eq:app-rear-inner}
 \sum_{d=0}^{h}A_{ap+d}
 \equiv p^2A_{ap}(\sigma_0+6a\sigma_1)\pmod {p^6}.
\end{equation}

Write $a=p-1-b$, $0\le b<h$.  Removing the unique factor $p$ from the
central binomial coefficient gives
\[
 \frac1p\binom{2(p-1-b)p}{(p-1-b)p}
 \equiv-\frac1{(2b+1)\binom{2b}{b}}\pmod p.
\]
Together with $64^{(p-1-b)p}\equiv64^{-b}\pmod p$, this gives
\begin{equation}\label{eq:app-rear-outer}
 \frac{A_{(p-1-b)p}}{p^3}
 \equiv-\frac1{(2b+1)^3A_b}\pmod p.
\end{equation}
Substitution of \eqref{eq:app-rear-inner} and
\eqref{eq:app-rear-outer} into \eqref{eq:app-rear-strata} gives
\[
 \frac{\mathcal R_p}{p^5}
 \equiv-\sum_{b=0}^{h-1}
 \frac{\sigma_0+6(p-1-b)\sigma_1}{(2b+1)^3A_b}\pmod p.
\]
Using
\[
 \frac{A_{b+1}}{A_b}=\frac{(2b+1)^3}{8(b+1)^3}
\]
and putting
\[
 R_3=\sum_{x=1}^{h}\frac1{x^3A_x},\qquad
 R_2=\sum_{x=1}^{h}\frac1{x^2A_x},
\]
we obtain the scalar reduction
\begin{equation}\label{eq:app-rear-scalar}
 \frac{\mathcal R_p}{p^5}
 \equiv-\frac18(\sigma_0R_3-6\sigma_1R_2)\pmod p.
\end{equation}

It remains to relate the reciprocal moments.  The ratio recurrence gives
\[
 \frac1{x^3A_x}=\frac8{(2x-1)^3A_{x-1}}.
\]
Put $y=h+1-x$.  Modulo $p$, anti-palindromy gives
\[
 A_y\equiv-A_{x-1},\qquad2x-1\equiv-2y,
\]
and hence
\[
 \frac1{x^3A_x}\equiv\frac1{y^3A_y}\pmod p.
\]
The involution $x\mapsto h+1-x$ therefore gives
\[
 2R_2\equiv
 \sum_{x=1}^{h}(x+y)\frac1{x^3A_x}
 =(h+1)R_3\equiv\frac12R_3\pmod p,
\]
so
\begin{equation}\label{eq:app-reciprocal-reflection}
 R_3\equiv4R_2\pmod p.
\end{equation}
Together with $\sigma_0\equiv3\sigma_1/2\pmod p$, this yields
\[
 \sigma_0R_3-6\sigma_1R_2
 \equiv(4\sigma_0-6\sigma_1)R_2=0\pmod p.
\]
Thus \eqref{eq:app-rear-scalar} proves
$\mathcal R_p\in p^6\Zp$, and
\begin{equation}\label{eq:app-origin}
 \sum_{k=0}^{p^2-1}A_k=p^2\pmod {p^5},\qquad
 (TL)(0)=1\pmod {p^3}.
\end{equation}
Combining \eqref{eq:app-TLconstant} with \eqref{eq:app-origin} proves
\begin{equation}\label{eq:app-all-points}
 E_0(a)=p^2\{(TL)(a)-1\}\in p^5\Zp
 \qquad(a\in\mathbb Z_{\ge0}).
\end{equation}

It remains to pass from point values to Taylor coefficients.  Write
$E_0=\sum t_na^n$.  By \cref{prop:bidisc},
$\vp(t_n)\ge\alpha_pn$, and
$6\alpha_p\ge6\alpha_7=29/7>4$; hence $t_n\in p^5\Zp$ for $n\ge6$.
The degree-five truncation consequently has values in $p^5\Zp$ at
$0,1,\ldots,5$.  Its evaluation matrix has determinant
$\prod_{j=1}^{5}j!$, a unit for $p\ge7$, so its first six coefficients
also belong to $p^5\Zp$.  This proves $E_0\in p^5\OO$.
The same two-node argument, using $2\alpha_p>1$, proves $S\in p^2\OO$.

Only two external one-block congruences used in this appendix come from
\cite{Tauraso}: equation \eqref{eq:app-tauraso-input} is his
Lemma~1, equation~(13), and \eqref{eq:app-half-harmonic} follows from
the half-harmonic congruence recorded in his Section~2, equation~(4).
The residue identities, the two-block calculation and the congruence
\eqref{eq:app-all-points} are proved here and are not attributed to
Tauraso's Theorem~3.

\section{Denominator control and the actual Green primitive}
\label{app:green-denominators}

There are two useful estimates for polynomial division by $D$.  The
elementary double-factorial estimate is sufficient for large primes.  The
discrete Green kernel has only logarithmic denominator loss and is essential
at $p=3$.

For $n\ge2$, define $Q_n,R_n$ by
\begin{equation}\label{eq:app-monomial-division}
 a^n=DQ_n+R_n,\qquad\deg R_n\le1.
\end{equation}
Since
\[
 D(a^j)=-(j+3/2)a^{j+2}+\text{terms of degree at most }j+1,
\]
descending elimination divides at most once by each pivot
$3/2,5/2,\ldots,(2n-1)/2$.  Every other coefficient of $D(a^j)$ is
$p$-integral.  Therefore
\[
 \|Q_n\|,\|R_n\|\le p^{B_p(n)},\qquad
 B_p(n)=\sum_{j=0}^{n-2}\vp(2j+3)=\vp((2n-1)!!).
\]
Finally,
\[
 B_p(n)=\vp((2n)!)-\vp(n!)
 \le\sum_{r\ge1}\frac{2n}{p^r}=\frac{2n}{p-1},
\]
which proves \cref{lem:double-factorial}.

We now prove the logarithmic estimate.  Put
$F_n=a\fall n$ and $P_k=(a-1)\fall k$.  The identities
$aP_k=(a-k)F_k$ and $P_k(a+1)=F_k(a)$ give
\[
 DP_k=-A_k^*F_k-B_k^*F_{k+1}-C_k^*F_{k+2},
\]
with the coefficients in \eqref{eq:three-term}.  Applying either quotient
coordinate to this identity shows that $G_n=n!g_n$ and $H_n=n!b_n$ solve
the associated homogeneous recurrence.  Their Casoratian satisfies
\[
 W_{n+1}=\frac{A_n^*}{C_n^*}W_n,\qquad W_0=1,
\]
and hence equals \eqref{eq:wronskian}.  Variation of constants now gives
\[
 c_{n,k}=-\frac{G_{k+1}H_n-H_{k+1}G_n}{C_k^*W_{k+1}}.
\]
The index $k+1$ is fixed by the inhomogeneous initial conditions: as a
function of $n$, this expression vanishes at $n=k+1$ and equals
$-1/C_k^*$ at $n=k+2$.  Thus it is precisely the coefficient of $P_k$
in the primitive of $F_n$.

Writing $m=k+1$ and substituting the Casoratian gives the cancellation
\begin{equation}\label{eq:app-green-cancelled}
 c_{n,k}=-\frac{2\cdot16^mn!}
 {m!\binom{2m}{m}^2}(g_mb_n-b_mg_n).
\end{equation}
The quotient $n!/m!$ is integral.  Kummer's carry formula gives
$\vp\binom{2m}{m}\le\lfloor\log_p(2m)\rfloor$, while
\eqref{eq:momentbound} controls the two moments.  Hence
\[
 \vp(c_{n,k})\ge-6\lfloor\log_p(2n+1)\rfloor.
\]
Both the $P_k$ and the Stirling change
$a^n=\sum_j S(n,j)F_j$, where $S(n,j)$ are the integral Stirling numbers
of the second kind, have integral coefficients, proving
\cref{lem:green-bound}.

For clarity, suppose $f=\sum t_na^n$ satisfies
$\vp(t_n)\ge\sigma n-C$, $\sigma>0$, and $\Pi f=0$.  On the disc
$|a|_p\le p^\delta$, $0<\delta<\sigma$, the terms of
$q_f=\sum_{n\ge2}t_nQ_n$ have norm at most
\[
 p^{-\sigma n+C+6\lfloor\log_p(2n+1)\rfloor+\delta(n-2)},
\]
which tends to zero.  Translation by one is bounded on this disc, so $D$
may be applied termwise.  The accumulated linear remainder is exactly
$\gamma(f)+\beta(f)u=0$.  Consequently $Dq_f=f$.
This proves existence of an actual overconvergent primitive; it neither
assumes nor proves that $D\Aalg$ is closed on the unit-radius Tate algebra.

\section{The finite certificate at three}
\label{app:p3}

Here we separate the theoretical infinite tail from the finite exact
calculation.  Write $E_0=(T^2-9)1=\sum t_na^n$.  The gamma estimate gives
\begin{equation}\label{eq:app-p3-slope}
 \vp(t_n)\ge\lceil n/6\rceil,
\end{equation}
and \cref{prop:spectrum} gives $\Pi E_0=0$.  Thus the preceding Green
construction gives an actual $q=\sum_{n\ge2}t_nQ_n$ with $Dq=E_0$.

For $199\le n\le363$,
$\lfloor\log_3(2n+1)\rfloor=5$ and $n/6-30>3$.  If
$\ell=\lfloor\log_3(2n+1)\rfloor\ge6$, then
\[
 n/6-6\ell\ge(3^\ell-1)/12-6\ell>3;
\]
the last expression is positive at $\ell=6$ and increases thereafter.
Since valuations are integral,
\begin{equation}\label{eq:app-p3-tail}
 \vp(t_nQ_n)\ge4\qquad(n\ge199).
\end{equation}
The lower bound tends to infinity, so the entire tail is a convergent
element of $81\OO$.  Nothing about this infinite range is delegated to a
computer.

It remains to certify $q_{\rm low}=\sum_{n=2}^{198}t_nQ_n$.  Let
\[
 \exp(z+z^3/3)=\sum_{n\ge0}d_nz^n,\qquad
 d_0=1,\quad nd_n=d_{n-1}+d_{n-3},\quad d_j=0\ (j<0).
\]
Then
\begin{equation}\label{eq:app-gamma3}
 G(z)=\Gamma_3(-3z)=\sum_{m\ge0}(-3)^md_{3m}z\fall m,\qquad
 \vp((-3)^md_{3m})\ge m/6.
\end{equation}
At precision $3^P$, omission of the terms $m\ge6P$ therefore gives an
error in $3^P\mathbb Z_3\langle z\rangle$.  The exact outer weight is
\begin{equation}\label{eq:app-V3}
 V(a)=\left(\frac{G(-a-1/2)}
 {G(-1/2)G(-a)(1+6a)}\right)^3.
\end{equation}
All denominators here are units.  Integral affine substitutions,
multiplication and inversion of units preserve the $3^P$ error ball.
The certificate constructs $V$, then $S=T1$, and then
\[
 E_0=V(a)\sum_{d=0}^{2}
 \left(\frac{(1/2+3a)_d}{(1+3a)_d}\right)^3S(3a+d)-9.
\]
It retains Taylor degree $<6P$, not merely degree $198$.  The slope
\eqref{eq:app-p3-slope} therefore ensures that substitution
$a\mapsto3a+d$ cannot return a discarded term to low degree with
uncontrolled precision.  For $0\le n\le198$, the computed coefficient
$\widetilde t_n$ satisfies
\begin{equation}\label{eq:app-tn-error}
 t_n-\widetilde t_n\in3^P\mathbb Z_3.
\end{equation}

Every $Q_n$, $n\le198$, is constructed by exact rational division, and
the identity \eqref{eq:app-monomial-division} is checked over $\mathbb Q$.
The actual maximum denominator loss in this finite range is five powers
of $3$.  Thus \eqref{eq:app-tn-error} loses at most five powers.  Even the
uniform logarithmic bound loses at most
$6\lfloor\log_3(397)\rfloor=30$ powers, so $P=34$ still leaves the four
powers needed.

At $P=34$, exact rational arithmetic followed by exact reduction modulo
$3^{34}$ shows coefficientwise that
\[
 \sum_{n=2}^{198}\widetilde t_nQ_n\in81\mathbb Z_3[a].
\]
Indeed every individual product $\widetilde t_nQ_n$ already has this
property, so no unexplained cancellation between different $n$ is used.
A second run at $P=38$ gives the same normalized coefficient vector modulo
$3^6$.  The error argument above, rather than agreement of the two runs,
proves that $q_{\rm low}\in81\OO$.  Together with
\eqref{eq:app-p3-tail}, this yields
\[
 (T^2-9)1=Dq,\qquad q\in81\mathbb Z_3\langle a\rangle.
\]
The program certifies only the finite Taylor range $2\le n\le198$; the
spectral theorem supplies the vanishing linear remainder, and the
infinite tail is proved symbolically above.

\section*{Statements and Declarations}

\paragraph{Funding.} The author received no funding for this work.

\paragraph{Competing interests.} The author declares no competing interests.

\paragraph{Data and code availability.} The exact PARI/GP certificate,
its fail-closed runner, and saved output are included as arXiv ancillary
files and mirrored at
\url{https://github.com/huiminZheng-collab/sun-conjecture-2-6-proof-candidate}.

\paragraph{Generative-AI disclosure.}
The proof presented in this article was found by OpenAI Codex.

\end{document}